\documentclass[11pt]{article}

\usepackage[a4paper,margin=1in]{geometry}
\usepackage{amsmath,amssymb,amsthm,mathtools}
\usepackage{enumitem}
\usepackage{booktabs}
\usepackage{graphicx}
\usepackage{hyperref}

\allowdisplaybreaks

\hypersetup{colorlinks=true,linkcolor=blue,citecolor=blue,urlcolor=blue,
pdftitle={Recovered-Coefficient Discontinuous Galerkin Approximation of an Averaged-State Credit-Rating Migration PDE},
pdfauthor={Davood Damircheli}}

\newtheorem{theorem}{Theorem}[section]
\newtheorem{lemma}[theorem]{Lemma}
\newtheorem{proposition}[theorem]{Proposition}

\newtheorem{assumption}[theorem]{Assumption}
\newtheorem{remark}[theorem]{Remark}
\newtheorem{definition}[theorem]{Definition}

\newcommand{\R}{\mathbb{R}}
\newcommand{\eps}{\varepsilon}

\newcommand{\Eh}{\mathcal{E}_h}
\newcommand{\Ehi}{\mathcal{E}_h^i}
\newcommand{\Ehb}{\mathcal{E}_h^b}
\newcommand{\jump}[1]{[\![#1]\!]}
\newcommand{\avg}[1]{\{\!\{#1\}\!\}}
\newcommand{\Ch}{\mathcal C_h}
\newcommand{\Osw}{\mathcal O_h}
\newcommand{\Sm}{\mathcal S_\ell}
\newcommand{\Lim}{\mathcal L_M}
\newcommand{\Ah}{\mathcal A_h}

\title{Recovered-Coefficient Discontinuous Galerkin Approximation of an Averaged-State Credit-Rating Migration PDE}
\author{Davood Damircheli\thanks{Department of Mathematics, Louisiana State University, Baton Rouge, LA 70803, USA.}}
\date{}

\begin{document}
\maketitle

\begin{abstract}
We develop and analyze a discontinuous Galerkin (DG) method for an averaged-state regularization of a credit-rating migration equation, in which the volatility switch is smoothed at scale $\eps$ and evaluated at a mollified state of width $\ell$.  The diffusion operator is written in conservative form and discretized by symmetric interior penalty, and the effective drift by upwinding.  The nonlinear coefficients are evaluated from a recovered state obtained by Oswald averaging, the model mollifier, and an explicit norm cutoff calibrated by a data-computable maximum bound.  The recovered coefficients are single-valued and uniformly $W^{2,\infty}$, so the standard upwind energy identity holds, and the recovery is Lipschitz from $L^2$ and compatible with the projected exact state at order $h^{p+1}$.  For fixed $(\eps,\ell)$ we prove well-posedness and a maximum bound for the continuous problem, an $O(\ell^2)$ distance to the pointwise-regularized model, stability of the frozen-coefficient backward Euler DG scheme, and an $O(h^p+\Delta t)$ error estimate in the DG energy norm; an exact Nitsche--lifting identity covers nonhomogeneous Dirichlet data.  Under local transversality, the regularized migration boundary is recovered at a proved, suboptimal rate.  Computations with the analyzed scheme confirm the predicted rates and report the quantities on which the analysis rests.
\end{abstract}

\section{Introduction}\label{sec:introduction}

Credit-rating migration models with volatility switching lead to nonlinear parabolic equations whose diffusion coefficient changes across a migration threshold.  In the corporate-bond model of Hu, Liang, and Wu \cite{HuLiangWu2015} and its successors \cite{WuLiang2020} the sharp switching law produces a free-boundary problem.  We study a regularization in which the discontinuous switch is replaced by a smooth transition $H_\eps$ and the state at which the switch is evaluated is mollified at a fixed scale $\ell$.  For fixed $(\eps,\ell)$ this yields a uniformly parabolic quasilinear equation whose coefficients depend on the solution through a smoothing, nonlocal map.  The parameter $\ell$ has the same modeling status as $\eps$, and we quantify the distance to the pointwise-regularized model (Proposition~\ref{prop:ell-modeling}); the sharp-interface limit $(\eps,\ell)\to(0,0)$ is not addressed.

We use a DG discretization in space and backward Euler in time.  DG provides an upwind treatment of the solution-dependent drift and local conservation, and it is the natural starting point for the two-dimensional and sharp-interface settings that we regard as future directions.  The diffusion is discretized by symmetric interior penalty (SIPG) \cite{Arnold2002,Riviere2008,DiPietroErn2012} after rewriting the equation in conservative form, and the effective drift by an upwind flux \cite{BrezziMariniSuli2004,CockburnShu2001}.  The structural difficulty is that the drift entering the upwind flux must be single-valued and Lipschitz at every face: if the coefficients were evaluated from the discontinuous DG state, the drift would have two traces per face, its derivative would contain Dirac masses, and the scalar upwind energy identity would fail.  We therefore evaluate the coefficients from a recovered state $\Ch w_h=\Lim(\Sm\Osw w_h)$ built from an Oswald averaging operator $\Osw$ \cite{KarakashianPascal2003,BurmanErn2007}, the mollifier $\Sm$ of the continuous model, and an explicit norm cutoff $\Lim$.  Because the scheme and the model use the same $\Sm$, the recovered exact state is the coefficient state of the continuous problem, and admissibility, $L^2$-Lipschitz stability, and projection compatibility of the recovery are lemmas rather than hypotheses.

\paragraph{Contributions.}
\begin{enumerate}[label=(\roman*),itemsep=1pt,topsep=2pt]
\item An explicit recovery operator $\Ch$, calibrated by a data-computable maximum bound, which makes the coefficient map globally Lipschitz from $L^2$ into $W^{2,\infty}$ and is compatible with the projected exact state at order $h^{p+1}$.
\item For fixed $(\eps,\ell)$: well-posedness, a maximum bound, and an $O(\ell^2)$ estimate relative to the pointwise-regularized equation; discrete well-posedness and stability; and an $O(h^p+\Delta t)$ DG-energy error estimate that, through an exact Nitsche--lifting identity, covers the nonhomogeneous boundary data of the migration problem.
\item A proved, suboptimal location estimate for the regularized migration boundary under local transversality, with the optimal rate identified as conditional.
\item Computations with the analyzed operator, including manufactured-solution and payoff-datum convergence studies and the run-specific diagnostics on which the analysis depends.
\end{enumerate}

The SIPG trace, inverse and coercivity tools follow \cite{Arnold2002,Riviere2008,DiPietroErn2012}, with the broken Poincar\'e--Friedrichs inequality of \cite{Brenner2003}; upwind DG for $W^{1,\infty}$ velocities goes back to \cite{BrezziMariniSuli2004,CockburnShu2001}; the Oswald operator is taken from \cite{KarakashianPascal2003,BurmanErn2007}; and the parabolic projection-splitting argument is in the tradition of \cite{Wheeler1973,DouglasDupont1970,Thomee2006}.  Once the coefficient map is globally Lipschitz in suitable norms, much of the error analysis is a Lipschitz perturbation of these established linear arguments.  The contribution of the paper is the construction of a recovery that simultaneously yields bounded coefficient derivatives, consistency with the continuous model, and projection compatibility, together with a complete convergence analysis for the specified averaged-state model.  It does not establish convergence to the sharp-interface financial model.

Section~\ref{sec:model} presents the model and its well-posedness, Section~\ref{sec:scheme} the recovery and the DG scheme, Section~\ref{sec:analysis} stability and the error estimate, Section~\ref{sec:boundary} boundary recovery, Section~\ref{sec:numerics} the computations, and Section~\ref{sec:conclusion} conclusions.

\section{The averaged-state model}\label{sec:model}

\subsection{Regularized switching and the mollified state}

Let $\Omega=(x_{\min},x_{\max})$ be a bounded log-price interval, $L=x_{\max}-x_{\min}$, and let $u(x,t)$ be the bond value in time-to-maturity variables.  The pointwise-regularized problem is
\begin{equation}\label{eq:pointwise-model}
\partial_t u-a_\eps(u,x)\,\partial_{xx}u-b_\eps(u,x)\,\partial_x u+ru=f\ \text{ in }\Omega\times(0,T],\qquad
u(\cdot,0)=u_0,\qquad u=g_D\ \text{ on }\partial\Omega,
\end{equation}
with $r\ge0$, threshold $\Psi(x)=\nu e^x$ ($\nu>0$), and
\begin{equation}\label{eq:coefficients}
\sigma_\eps(z,x)=\sigma_H+(\sigma_L-\sigma_H)H_\eps\bigl(z-\Psi(x)\bigr),\qquad
a_\eps=\tfrac12\sigma_\eps^2,\qquad b_\eps=r-a_\eps,\qquad 0<\sigma_H<\sigma_L .
\end{equation}
The smooth switch satisfies $0\le H_\eps\le1$ and $|H_\eps^{(k)}|\le C\eps^{-k}$ for $k=1,2,3$; for example $H_\eps(s)=\frac12(1+\tanh(s/\eps))$.  Hence $a_{\min}:=\frac12\sigma_H^2\le a_\eps\le a_{\max}:=\frac12\sigma_L^2$ and $\|\partial_z^ka_\eps\|_{L^\infty}+\|\partial_z^kb_\eps\|_{L^\infty}\le C\eps^{-k}$ for $k\le3$.  The third derivative is needed only to difference the effective reaction coefficient $\gamma$ below.

Let $E$ be a linear extension operator from $\Omega$ to $\R$, bounded from $L^2(\Omega)$ to $L^2(\R)$ (constant $C_{E,0}$), from $H^m(\Omega)$ to $H^m(\R)$ for $m\le3$, and from $W^{m,\infty}(\Omega)$ to $W^{m,\infty}(\R)$ for $m\le2$.  A concrete choice, used in Section~\ref{sec:numerics}, is the three-term Hestenes reflection \cite{Hestenes1941} $Ez(x_{\max}+s)=\sum_{j=1}^3c_jz(x_{\max}-js)$, $c=(6,-8,3)$, with the mirror formula at $x_{\min}$, multiplied by a fixed cutoff $\chi$ with $0\le\chi\le1$, $\chi=1$ within distance $L/8$ of $\bar\Omega$, and $\chi=0$ beyond distance $L/4$ (so that only values in $\Omega$ are reflected); it matches derivatives of order $\le2$ and $C_{E,0}\le(1+2(\sum_j|c_j|j^{-1/2})^2)^{1/2}$.  Let $\rho\in C_c^\infty(-1,1)$ be even, nonnegative, with $\int\rho=1$, and $\rho_\ell(x)=\ell^{-1}\rho(x/\ell)$.  For $0<\ell\le\min(1,L/8)$ set
\begin{equation}\label{eq:mollifier}
\Sm z:=(\rho_\ell*Ez)\big|_\Omega .
\end{equation}
The averaged-state model analyzed in this paper is
\begin{equation}\label{eq:model}
\partial_t u-a_\eps(\Sm u,x)\,\partial_{xx}u-b_\eps(\Sm u,x)\,\partial_x u+ru=f,\qquad
u(\cdot,0)=u_0,\qquad u=g_D\ \text{ on }\partial\Omega,
\end{equation}
where $\Sm u=\Sm u(\cdot,t)$ acts in space.  A volatility regime that responds to a locally averaged rating variable is a modeling choice in its own right; the next proposition shows that, for fixed $\eps$, \eqref{eq:model} is $O(\ell^2)$-close to \eqref{eq:pointwise-model}.

\begin{proposition}[Distance to the pointwise model]\label{prop:ell-modeling}
Fix $\eps>0$, let $u_\ell$ and $u^\sharp$ solve \eqref{eq:model} and \eqref{eq:pointwise-model} with the same data, and assume $u^\sharp\in L^\infty(0,T;W^{2,\infty}(\Omega))$.  Then $\sup_{t\le T}\|u_\ell(t)-u^\sharp(t)\|_{L^\infty(\Omega)}\le C_\eps\,\ell^2$, where $C_\eps$ depends on $\eps$, $T$, $E$, and $\|u^\sharp\|_{L^\infty W^{2,\infty}}$ but not on $\ell$.
\end{proposition}
\begin{proof}
The difference $w=u_\ell-u^\sharp$ vanishes on the parabolic boundary and satisfies $\partial_tw-a_\eps(\Sm u_\ell)w_{xx}-b_\eps(\Sm u_\ell)w_x+rw=R$ with
$R=(a_\eps(\Sm u_\ell)-a_\eps(u^\sharp))\partial_{xx}u^\sharp+(b_\eps(\Sm u_\ell)-b_\eps(u^\sharp))\partial_xu^\sharp$.
Since $\|\Sm v\|_{L^\infty}\le C\|v\|_{L^\infty}$ and, by Taylor's formula, the evenness of $\rho$, and the $W^{2,\infty}$ bound of $E$, $\|\Sm v-v\|_{L^\infty(\Omega)}\le\frac{\ell^2}2\int t^2\rho\,\|(Ev)''\|_{L^\infty(\R)}\le C\ell^2\|v\|_{W^{2,\infty}}$, we get $\|R(t)\|_{L^\infty}\le C\eps^{-1}K(\|w(t)\|_{L^\infty}+\ell^2K)$ with $K=\|u^\sharp\|_{L^\infty W^{2,\infty}}$.  The maximum principle on a slab $[t_0,t_0+\tau]$ gives $\sup_{\rm slab}\|w\|_{L^\infty}\le\|w(t_0)\|_{L^\infty}+\tau\sup_{\rm slab}\|R\|_{L^\infty}$.  Choosing $\tau$ with $C\eps^{-1}K\tau\le\frac12$ and iterating over $\lceil T/\tau\rceil$ slabs proves the claim, with $C_\eps$ growing at most exponentially in $T\eps^{-1}K$.
\end{proof}

The second-order rate uses that $E$ preserves $W^{2,\infty}$; an extension bounded only in $W^{1,\infty}$ gives $O(\ell)$.  The estimate is structural: in the computations we take $\ell=\eps$, and we do not use Proposition~\ref{prop:ell-modeling} as a small-error statement there.

\subsection{Conservative form and lifting}

For a state $w$ write $\tilde w=\Sm w$ (later $\tilde w=\Ch w$) and define
\begin{equation}\label{eq:frozen-coefficients}
A_w=a_\eps(\tilde w,x),\qquad
\beta_w=\partial_za_\eps(\tilde w,x)\bigl(\partial_x\tilde w-\Psi'\bigr)-b_\eps(\tilde w,x),\qquad
\gamma_w=\partial_x\beta_w,\qquad b_w=b_\eps(\tilde w,x).
\end{equation}
Since $a_\eps$ depends on $(z,x)$ only through $z-\Psi(x)$, $\partial_xA_w=\partial_za_\eps(\tilde w,x)(\partial_x\tilde w-\Psi')$, so $\beta_w=\partial_xA_w-b_w$ and \eqref{eq:model} is equivalent, for strong solutions, to the conservative form
\begin{equation}\label{eq:conservative}
\partial_tu-\partial_x(A_u\partial_xu)+\beta_u\partial_xu+ru=f .
\end{equation}
By the chain rule, $\gamma_w=\partial_z^2a_\eps\,s_x^2+\partial_za_\eps(s_{xx}+s_x)$ with $s=\tilde w-\Psi$, which involves $\partial_x^2\tilde w$; differencing it in $\tilde w$ involves $H_\eps'''$.  Let
\begin{equation}\label{eq:lifting}
G(x,t)=g_D(x_{\min},t)\frac{x_{\max}-x}{L}+g_D(x_{\max},t)\frac{x-x_{\min}}{L},\qquad g_G(t):=\partial_xG .
\end{equation}
Because $\partial_{xx}G=0$ and $\partial_xA_u-\beta_u=b_u$, the function $\hat u=u-G$ vanishes on $\partial\Omega$ and solves
\begin{equation}\label{eq:lifted}
\partial_t\hat u-\partial_x(A_u\partial_x\hat u)+\beta_u\partial_x\hat u+r\hat u
=f-\partial_tG-rG+g_G\,b_u .
\end{equation}
The coefficient state remains the full state $u$.  We define the regularized migration boundary as the level set $\{u-\Psi=0\}$ of the solution itself; under transversality it differs from the switching set $\{\Sm u=\Psi\}$ by $O(\ell^2)$.

\subsection{Well-posedness and regularity}

\begin{proposition}[Maximum bound]\label{prop:max-bound}
Let $r\ge0$ and $u_0,g_D,f$ be bounded.  Every bounded strong solution of \eqref{eq:model} satisfies
$\|u\|_{L^\infty(\Omega\times(0,T))}\le B_\infty:=\max(\|u_0\|_{L^\infty},\|g_D\|_{L^\infty})+T\|f\|_{L^\infty}$,
hence $\sup_t\|u(t)\|_{L^2(\Omega)}\le B_{\rm data}:=|\Omega|^{1/2}B_\infty$.
\end{proposition}
\begin{proof}
For the given solution the coefficients are bounded and measurable with $a_\eps\ge a_{\min}$.  Since $r\ge0$, $\pm(\max(\|u_0\|_{L^\infty},\|g_D\|_{L^\infty})+t\|f\|_{L^\infty})$ are super- and subsolutions, and the weak maximum principle applies.
\end{proof}

\begin{proposition}[Well-posedness at fixed regularization]\label{prop:wellposed}
Let $u_0\in H^1(\Omega)$, $f\in L^\infty(\Omega\times(0,T))$, and let the Dirichlet data be bounded, compatible with $u_0$ at the corners, and have an affine lifting $G\in H^1(0,T;H^1(\Omega))$.  Then \eqref{eq:model} has a unique solution with $u-G\in C([0,T];H_0^1(\Omega))\cap L^2(0,T;H^2(\Omega))$ and $\partial_tu\in L^2(0,T;L^2(\Omega))$.
\end{proposition}
\begin{proof}
Let $T_B(s)=\max(-B_\infty,\min(s,B_\infty))$ and consider the truncated problem in which $\Sm u$ is replaced by $\Sm T_Bu$.  By Lemma~\ref{lem:mollifier} below, $v\mapsto\Sm T_Bv$ is globally Lipschitz from $L^2$ into $W^{2,\infty}$ with range bounded in $W^{2,\infty}$, so for every $v\in C([0,T];L^2)$ the frozen linear problem has a unique solution $y=\mathcal Tv$ in the stated class with $\|y\|_{L^\infty H^1}+\|y\|_{L^2H^2}\le C$ independently of $v$.  For $z=\mathcal Tv_1-\mathcal Tv_2$, testing the difference of the conservative equations with $z$ gives
$\frac{d}{dt}\|z\|_{L^2}^2+a_{\min}\|z_x\|_{L^2}^2\le C\|v_1-v_2\|_{L^2}^2\|\partial_x\mathcal Tv_2\|_{L^2}^2+C\|z\|_{L^2}^2$,
so $\mathcal T$ is a contraction on $C([0,T_*];L^2)$ for small $T_*$, and continuation over slabs of length $T_*$ gives a unique truncated solution on $[0,T]$.  The proof of Proposition~\ref{prop:max-bound} uses only bounded measurable coefficients with $a\ge a_{\min}$ and therefore applies to it, so $T_Bu=u$ and $u$ solves \eqref{eq:model}.  Conversely, every solution in the stated class is bounded (one dimension), satisfies the maximum bound, and therefore solves the truncated problem; uniqueness follows.
\end{proof}

The error analysis requires more regularity, which we take as a hypothesis.

\begin{assumption}[Data, compatibility, and regularity]\label{ass:regularity}
The data satisfy $u_0\in H^{p+1}(\Omega)$, $f\in L^\infty(\Omega\times(0,T))\cap C([0,T];L^2(\Omega))$, $g_D(x_{\min},\cdot),g_D(x_{\max},\cdot)\in W^{2,\infty}(0,T)$, and the parabolic compatibility conditions at $\partial\Omega\times\{0\}$ up to the order required by the following regularity; in particular $u_0=g_D(\cdot,0)$ and $\partial_tg_D=a_\eps(\Sm u_0)u_0''+b_\eps(\Sm u_0)u_0'-ru_0+f$ at the two corners.  The solution satisfies
\begin{equation}\label{eq:regularity}
u\in L^\infty(0,T;H^{p+1}(\Omega)),\qquad\partial_tu,\ \partial_{tt}u\in L^2(0,T;L^2(\Omega)),
\end{equation}
and we write $\|u\|_{\mathcal Z}$ for the sum of these norms and $\|u_0\|_{H^{p+1}}$.
\end{assumption}

\begin{remark}\label{rem:regularity}
Assumption~\ref{ass:regularity} is expected from classical quasilinear theory \cite{LSU1968,Lieberman1996}: the coefficient map is smoothing, and the maximum bound keeps the coefficient norms bounded a priori.  We do not prove it.  The data norms are not uniform in $\eps$; e.g.\ $\|u_0\|_{H^{p+1}}\simeq\eps^{-p-1/2}$ for a datum with an $\eps$-layer.  Zeroth-order matching $u_0=g_D(\cdot,0)$ alone does not provide \eqref{eq:regularity}; the higher conditions are needed.  The migration datum of Section~\ref{sec:numerics} satisfies the conditions of every order exactly (Remark~\ref{rem:payoff-compat}).
\end{remark}

\section{Recovery operator and DG scheme}\label{sec:scheme}

\subsection{Mesh, spaces, and norms}

Let $x_{\min}=x_0<\dots<x_{N_h}=x_{\max}$, $K_j=(x_{j-1},x_j)$, $h_j=|K_j|$, $h=\max_jh_j$, with $h_j/h_{j\pm1}\le\varrho$.  Let $\Ehi$ and $\Ehb$ be the interior and boundary nodes, $\Eh=\Ehi\cup\Ehb$, $h_e=\min$ of the adjacent element lengths, and $V_h=\{v\in L^2(\Omega):v|_K\in\mathbb P_p(K)\}$, $p\ge1$.  At an interior node $\jump v=v^--v^+$ and $\avg v=\frac12(v^-+v^+)$; at a boundary node with outward normal $n_e$, $\jump v=v\,n_e$ and $\avg v=v$.  With $\eta_e>0$ the penalty parameter, set
\[
\|v\|_{DG}^2=\|v\|_{L^2}^2+|v|_{1,h}^2+|v|_{J,h}^2,\qquad
|v|_{1,h}^2=\sum_K\|\partial_xv\|_{L^2(K)}^2,\qquad
|v|_{J,h}^2=\sum_{e\in\Eh}\frac{\eta_e}{h_e}\jump v_e^2,
\]
$\|v\|_{A,h}^2=\sum_K\int_KA|\partial_xv|^2+\sum_e\eta_eA(e)h_e^{-1}\jump v_e^2$ for a continuous weight $A>0$, $|v|_{\beta,h}^2=\sum_e|\beta(e)|\jump v_e^2$, and, for piecewise $H^2$ functions, $\|z\|_{DG,{\rm ext}}^2=\|z\|_{DG}^2+\sum_eh_e|\avg{\partial_xz}_e|^2+\sum_eh_e^{-1}\jump z_e^2$.  We use the standard trace and inverse inequalities \cite{Riviere2008,DiPietroErn2012} and the local $L^2$ projection $\Pi_h$, which reproduces $\mathbb P_p$ elementwise and satisfies
\begin{equation}\label{eq:approx}
\|u-\Pi_hu\|_{L^2}+h\|u-\Pi_hu\|_{DG,{\rm ext}}+h^{1/2}\Bigl(\sum_e\jump{u-\Pi_hu}_e^2\Bigr)^{1/2}\le Ch^{p+1}|u|_{H^{p+1}} .
\end{equation}

\subsection{Recovery operator}

\begin{definition}[Oswald averaging]\label{def:oswald}
For $v\in V_h$ let $\delta_{K,j}$ be the difference between the average of the two traces at the node $x_j\in\Ehi\cap\partial K$ and the trace of $v|_K$ there ($\delta_{K,j}=0$ at boundary nodes), and let $\lambda_{K,j}$ be the affine function on $K$ equal to $1$ at $x_j$ and $0$ at the other endpoint.  Then $\Osw v:=v+\sum_K\sum_{x_j\in\partial K}\delta_{K,j}\lambda_{K,j}\in V_h\cap C^0(\bar\Omega)$.  On continuous functions we set $\Osw v=v$.
\end{definition}

Since $|\delta_{K,j}|=\frac12|\jump v_{x_j}|$ and $\|\lambda_{K,j}\|_{L^2(K)}^2=h_K/3$,
\begin{equation}\label{eq:oswald}
\|\Osw v-v\|_{L^2(K)}^2\le\tfrac{h_K}{3}\sum_{x_j\in\partial K\cap\Ehi}\jump v_{x_j}^2,\qquad
\|\Osw v\|_{L^2(K)}\le C\|v\|_{L^2(\omega_K)},
\end{equation}
where $\omega_K$ is $K$ with its neighbors; the second bound follows from the first and the inverse trace inequality \cite{KarakashianPascal2003}.

\begin{lemma}[Mollifier bounds]\label{lem:mollifier}
For $z\in L^2(\Omega)$ and $0\le k\le3$,
$\|\partial_x^k\Sm z\|_{L^\infty}\le C\ell^{-k-1/2}\|z\|_{L^2}$ and
$\|\Sm z\|_{H^3}\le C_{S,3}\ell^{-3}\|z\|_{L^2}$ with $C_{S,3}=C_{E,0}\bigl(\sum_{k\le3}\|\rho^{(k)}\|_{L^1}^2\bigr)^{1/2}$.
\end{lemma}
\begin{proof}
$|\partial_x^k(\rho_\ell*Ez)|\le\|\rho_\ell^{(k)}\|_{L^2}\|Ez\|_{L^2}$ with $\|\rho_\ell^{(k)}\|_{L^2}=\ell^{-k-1/2}\|\rho^{(k)}\|_{L^2}$, and $\|\rho_\ell^{(k)}*Ez\|_{L^2}\le\ell^{-k}\|\rho^{(k)}\|_{L^1}\|Ez\|_{L^2}$ by Young's inequality; use $\ell\le1$.
\end{proof}

Fix $\theta\in C^\infty(\R)$ with $0\le\theta\le1$, $\theta=1$ on $(-\infty,\frac12]$, $\theta=0$ on $[1,\infty)$, and $C_\theta=\|\theta'\|_{L^\infty}$.  With $B_{\rm data}$ from Proposition~\ref{prop:max-bound}, define
\begin{equation}\label{eq:recovery}
\Lim v:=\theta\Bigl(\frac{\|v\|_{H^3(\Omega)}}{M}\Bigr)v,\qquad
M:=2C_{S,3}\ell^{-3}(B_{\rm data}+1),\qquad
\Ch w:=\Lim(\Sm\Osw w),
\end{equation}
for $w\in V_h\cup C^0(\bar\Omega)$.  $M$ is computable from $u_0,g_D,f$ alone, and $\Ch$ is a single operator acting on discrete and continuous states.

\begin{lemma}[Properties of the recovery]\label{lem:recovery}
\begin{enumerate}[label=(\alph*),itemsep=1pt,topsep=2pt]
\item\label{it:lim} $\|\Lim v\|_{H^3}\le M$, $\|\Lim v_1-\Lim v_2\|_{H^3}\le(1+C_\theta)\|v_1-v_2\|_{H^3}$, and $\Lim v=v$ if $\|v\|_{H^3}\le M/2$.
\item\label{it:inactive} If $\|\Osw w\|_{L^2}\le B_{\rm data}+1$ then $\Ch w=\Sm\Osw w$.  In particular, $\Ch u(t)=\Sm u(t)$ for the solution of \eqref{eq:model}: the recovered exact state is the coefficient state of the continuous problem.
\item\label{it:admissible} For every $w$, $\|\Ch w\|_{W^{2,\infty}}\le C_{\rm Sob}M$.  Consequently $a_{\min}\le A_w\le a_{\max}$, $A_w,\beta_w\in W^{1,\infty}(\Omega)$ have single-valued traces, and $\|\beta_w\|_{L^\infty}+\|\gamma_w\|_{L^\infty}\le C_{\beta\gamma}=C(\eps,\ell,{\rm data})$, uniformly in $w$, $h$, and $\Delta t$.
\end{enumerate}
\end{lemma}
\begin{proof}
\ref{it:lim}: the bound and inactivity are immediate.  For the Lipschitz bound, if $\|v_1\|_{H^3},\|v_2\|_{H^3}\ge M$ both images vanish; otherwise, say $\|v_2\|_{H^3}\le M$, write $\Lim v_1-\Lim v_2=\theta_1(v_1-v_2)+(\theta_1-\theta_2)v_2$, with $\theta_i=\theta(\|v_i\|_{H^3}/M)$, and use $|\theta_1-\theta_2|\le C_\theta M^{-1}\|v_1-v_2\|_{H^3}$.  \ref{it:inactive}: Lemma~\ref{lem:mollifier} gives $\|\Sm\Osw w\|_{H^3}\le M/2$; for $w=u(t)$ use $\Osw u=u$ and Proposition~\ref{prop:max-bound}.  \ref{it:admissible}: \ref{it:lim}, the embedding $H^3(\Omega)\hookrightarrow W^{2,\infty}(\Omega)$ \cite{AdamsFournier2003}, and the coefficient bounds of Section~\ref{sec:model}.
\end{proof}

The limiter removes a circularity: without it, a $W^{2,\infty}$ bound for $\Sm\Osw U_h^n$ would follow only from discrete stability, whose constant depends exponentially on the G\aa rding constant, which in turn depends on that bound.  The analyzed scheme is the limited one; it coincides with the unlimited scheme $\Sm\Osw$ at every step where the computable margin $M/2-\|\Sm\Osw U_h^n\|_{H^3}$ is nonnegative.  This margin is reported in Section~\ref{sec:numerics}; it is a property of a run, not a theorem.

\subsection{DG forms and the Nitsche--lifting identity}

For a state $w$ let $A_w,\beta_w,\gamma_w,b_w$ be given by \eqref{eq:frozen-coefficients} with $\tilde w=\Ch w$, and write $A_{w,e}=A_w(e)$, $\beta_{w,e}=\beta_w(e)$.  The SIPG and upwind forms are
\begin{align*}
a_h(w;u,v)&=\sum_K\!\int_KA_w\partial_xu\,\partial_xv
-\sum_{e\in\Eh}\bigl(\avg{A_w\partial_xu}_e\jump v_e+\avg{A_w\partial_xv}_e\jump u_e\bigr)\\
&\quad+\sum_{e\in\Eh}\frac{\eta_eA_{w,e}}{h_e}\jump u_e\jump v_e,\\
c_h(w;u,v)&=-\sum_K\!\int_K\bigl(\beta_wu\,\partial_xv+\gamma_wuv\bigr)+\sum_{e\in\Eh}\widehat{\beta_wu}_e\jump v_e,
\end{align*}
where $\widehat{\beta_wu}_e=\beta_{w,e}\avg u_e+\frac12|\beta_{w,e}|\jump u_e$ at interior nodes, $\widehat{\beta_wu}_e=\beta_{w,e}u(e)$ at outflow nodes ($\beta_{w,e}n_e\ge0$), and $\widehat{\beta_wu}_e=0$ at inflow nodes.  Let $\Ah(w;u,v)=a_h+c_h+r(u,v)$.  The direct Nitsche load is
\begin{equation}\label{eq:nitsche-load}
\mathcal F_h^D(w;v;t)=(f,v)+\sum_{e\in\Ehb}\Bigl(-A_{w,e}g_D\,\partial_xv(e)\,n_e+\frac{\eta_eA_{w,e}}{h_e}g_D\,v(e)\Bigr)-\sum_{e\in\Ehb,\ \beta_{w,e}n_e<0}\beta_{w,e}n_eg_D\,v(e).
\end{equation}
Elementwise integration by parts shows that a smooth function $y$ with $y=g_D$ on $\partial\Omega$ satisfies
\begin{equation}\label{eq:consistency}
(\partial_ty,v)+\Ah(w;y,v)=\mathcal F_h^D(w;v;t)\quad\forall v\in V_h\quad\text{whenever}\quad
\partial_ty-\partial_x(A_wy_x)+\beta_wy_x+ry=f,
\end{equation}
and the same identity with $g_D=0$ holds for functions vanishing on $\partial\Omega$.

\begin{lemma}[Nitsche--lifting identity]\label{lem:lifting}
For every state $w$ and every $t$, with $G$ from \eqref{eq:lifting},
\[
\mathcal F_h^D(w;v;t)-(\partial_tG,v)-\Ah(w;G,v)=\bigl(f-\partial_tG-rG+g_G\,b_w,\;v\bigr)\qquad\forall v\in V_h .
\]
\end{lemma}
\begin{proof}
Apply \eqref{eq:consistency} to $y=G$ with source $\partial_tG-\partial_x(A_wg_G)+\beta_wg_G+rG=\partial_tG+rG-g_Gb_w$, using $\partial_xA_w-\beta_w=b_w$.  All boundary terms cancel.
\end{proof}

\subsection{Fully discrete scheme}

Let $0=t_0<\dots<t_N=T$, $\Delta t_n=t_n-t_{n-1}$, $\Delta t=\max_n\Delta t_n$, $\delta_tz^n=(z^n-z^{n-1})/\Delta t_n$.  Given $U_h^0\in V_h$, the frozen-coefficient backward Euler scheme is: find $U_h^n\in V_h$ with
\begin{equation}\label{eq:scheme}
(\delta_tU_h^n,v)+\Ah(U_h^{n-1};U_h^n,v)=\mathcal F_h^D(U_h^{n-1};v;t_n)\qquad\forall v\in V_h .
\end{equation}
Since $p\ge1$, $G^n=G(\cdot,t_n)\in V_h$.  By Lemma~\ref{lem:lifting} (with $\delta_tG^n$ in place of $\partial_tG$), $\widehat U_h^n=U_h^n-G^n$ satisfies exactly
\begin{equation}\label{eq:lifted-scheme}
(\delta_t\widehat U_h^n,v)+\Ah(U_h^{n-1};\widehat U_h^n,v)=(\hat f^n(U_h^{n-1}),v),\qquad
\hat f^n(w):=f^n-\delta_tG^n-rG^n+g_G(t_n)\,b_w,
\end{equation}
a scheme with homogeneous boundary and inflow data and a volume load whose dependence on the state is through $b_w$ only, so no boundary factor $h^{-1/2}\|g_D\|$ enters the analysis.

\section{Stability and error analysis}\label{sec:analysis}

Throughout, $(\eps,\ell)$ are fixed, and $C$ denotes constants depending on $\eps,\ell,T,p,\varrho$, the data, and $\|u\|_{\mathcal Z}$, but not on $h$, $\Delta t$, or $n$.  Negative powers of $\eps$ and $\ell$ enter through Lemmas~\ref{lem:mollifier} and \ref{lem:recovery}; we do not track them.

\subsection{Coercivity and the upwind identity}

\begin{assumption}[Penalty]\label{ass:penalty}
$\eta_e\ge\eta_0p^2$ with $\eta_0\ge8a_{\max}/a_{\min}$.
\end{assumption}

\begin{lemma}[Coercivity and G\aa rding inequality]\label{lem:garding}
Under Assumption~\ref{ass:penalty}, for every state $w$ and $v\in V_h$:
\begin{enumerate}[label=(\alph*),itemsep=1pt,topsep=2pt]
\item $a_h(w;v,v)\ge\frac12\|v\|_{A_w,h}^2$;
\item $c_h(w;v,v)=\frac12|v|_{\beta_w,h}^2-\frac12\int_\Omega\gamma_wv^2$;
\item $\Ah(w;v,v)\ge c_0\|v\|_{DG}^2-C_0\|v\|_{L^2}^2+\frac12|v|_{\beta_w,h}^2$ with $c_0=\frac{a_{\min}}{2(1+C_P)}$ and $C_0=\max(0,\frac12C_{\beta\gamma}-r)$, where $C_P$ is the broken Poincar\'e constant, $\|v\|_{L^2}^2\le C_P(|v|_{1,h}^2+|v|_{J,h}^2)$ \cite{Brenner2003}.
\end{enumerate}
The constant $C_0$ is universal: it depends on the cutoff bound of Lemma~\ref{lem:recovery}\ref{it:admissible}, not on $w$.
\end{lemma}
\begin{proof}
(a) Since $A_w$ is continuous, $|\avg{A_w\partial_xv}_e|^2/A_{w,e}\le\frac{a_{\max}}2(|\partial_xv^-|^2+|\partial_xv^+|^2)$ at interior nodes and $\le a_{\max}|\partial_xv|^2$ at boundary nodes.  With the sharp endpoint inequality $h_K|q(e)|^2\le p^2\|q\|_{L^2(K)}^2$ for $q\in\mathbb P_{p-1}(K)$, $h_e\le h_K$, and two endpoints per element,
$\sum_eh_e|\avg{A_w\partial_xv}_e|^2/A_{w,e}\le C_Fp^2\sum_K\int_KA_w|\partial_xv|^2$ with $C_F=2a_{\max}/a_{\min}$.  Young's inequality with weight $(2C_Fp^2)^{-1}$ bounds the consistency terms by $\frac12\sum_K\int_KA_w|\partial_xv|^2+2C_Fp^2\sum_eA_{w,e}h_e^{-1}\jump v_e^2$, and the last term is at most half the penalty term when $\eta_e\ge4C_Fp^2$.  (b) Integrate $-\int_K\beta_wv\partial_xv=-\frac12\int_K\beta_w\partial_x(v^2)$ by parts; because $\beta_w$ has a single trace, the element-boundary terms combine with the upwind flux into $\frac12|\beta_{w,e}|\jump v_e^2$ at every node.  (c) Add (a), (b), and the reaction term; use $\|v\|_{A_w,h}^2\ge a_{\min}(|v|_{1,h}^2+|v|_{J,h}^2)\ge a_{\min}(1+C_P)^{-1}\|v\|_{DG}^2$ and $\|\gamma_w\|_{L^\infty}\le C_{\beta\gamma}$.
\end{proof}

\subsection{Lipschitz dependence on the state}

For two states define
$d^\sharp(w_1,w_2):=\|\Ch w_1-\Ch w_2\|_{W^{1,\infty}}+\|\gamma_{w_1}-\gamma_{w_2}\|_{L^\infty}$.

\begin{lemma}[Recovery distance]\label{lem:distance}
\begin{enumerate}[label=(\alph*),itemsep=1pt,topsep=2pt]
\item\label{it:lip-form} For $v\in V_h$ and $z\in V_h$ or $z$ piecewise $H^2$,
$|\Ah(w_1;z,v)-\Ah(w_2;z,v)|\le C\,d^\sharp(w_1,w_2)\|z\|_{DG,{\rm ext}}\|v\|_{DG}$, and $\|b_{w_1}-b_{w_2}\|_{L^2}\le C\,d^\sharp(w_1,w_2)$.
\item\label{it:dist} $d^\sharp(w_1,w_2)\le C_d\|\Osw w_1-\Osw w_2\|_{L^2}$ with $C_d=C(\eps,\ell,{\rm data})$.  In particular $d^\sharp(w_1,w_2)\le C\|w_1-w_2\|_{L^2}$ for $w_1,w_2\in V_h$.
\item\label{it:compat} If $u(t)\in H^{p+1}(\Omega)$, then $d^\sharp(\Pi_hu(t),u(t))\le Ch^{p+1}|u(t)|_{H^{p+1}}$.
\end{enumerate}
\end{lemma}
\begin{proof}
\ref{it:lip-form}: differences of $A_w$, $\beta_w$, $|\beta_w|$, $(\beta_wn_e)^+$, and $\gamma_w$ enter the forms coefficientwise, and $x\mapsto|x|$, $x\mapsto x^+$ are $1$-Lipschitz; and $\|A_{w_1}-A_{w_2}\|_{L^\infty}+\|\beta_{w_1}-\beta_{w_2}\|_{L^\infty}\le C\|\Ch w_1-\Ch w_2\|_{W^{1,\infty}}$ by \eqref{eq:frozen-coefficients} and Lemma~\ref{lem:recovery}\ref{it:admissible}; apply Cauchy--Schwarz elementwise and facewise.  \ref{it:dist}: by Lemma~\ref{lem:recovery}\ref{it:lim} and Lemma~\ref{lem:mollifier},
$\|\Ch w_1-\Ch w_2\|_{W^{2,\infty}}\le C_{\rm Sob}(1+C_\theta)C_{S,3}\ell^{-3}\|\Osw w_1-\Osw w_2\|_{L^2}$.
The chain-rule expression of $\gamma_w$ is Lipschitz in $(\tilde w,\partial_x\tilde w,\partial_x^2\tilde w)$ on the bounded set of Lemma~\ref{lem:recovery}\ref{it:admissible}, with a constant involving $H_\eps'''$; hence $\|\gamma_{w_1}-\gamma_{w_2}\|_{L^\infty}\le C\|\Ch w_1-\Ch w_2\|_{W^{2,\infty}}$.  The discrete case uses \eqref{eq:oswald}.  \ref{it:compat}: $\Osw u=u$, and by \eqref{eq:oswald} and the continuity of $u$, $\|\Osw\Pi_hu-u\|_{L^2}\le\|\Osw\Pi_hu-\Pi_hu\|_{L^2}+\|\Pi_hu-u\|_{L^2}\le C(\sum_eh_e\jump{\Pi_hu-u}_e^2)^{1/2}+Ch^{p+1}|u|_{H^{p+1}}$; conclude with \eqref{eq:approx} and \ref{it:dist}.
\end{proof}

Part \ref{it:compat} is where the choice of the same $\Sm$ in model and scheme matters.  For the pointwise model \eqref{eq:pointwise-model}, whose coefficient state is $u$ itself, any fixed-scale recovery satisfies $\Ch\Pi_hu\to\Sm u\ne u$, and the analogue of \ref{it:compat} saturates at the smoothing bias.

\subsection{Discrete stability}

\begin{lemma}[Two-level discrete Gr\"onwall inequality]\label{lem:gronwall}
If $E_n,D_n,F_n\ge0$ satisfy $E_n-E_{n-1}+\Delta t_nD_n\le a\Delta t_nE_n+b\Delta t_nE_{n-1}+\Delta t_nF_n$ with $a,b\ge0$ and $a\Delta t<1$, then $E_m+\sum_{n\le m}\Delta t_nD_n\le\exp\bigl((a+b)t_m/(1-a\Delta t)\bigr)\bigl(E_0+\sum_{n\le m}\Delta t_nF_n\bigr)$.
\end{lemma}
\begin{proof}
Let $q_n=(1+b\Delta t_n)/(1-a\Delta t_n)$ and $P_n=\prod_{j\le n}q_j$.  Dividing by $1-a\Delta t_n$ gives $E_n+\Delta t_nD_n\le q_nE_{n-1}+\Delta t_nF_n/(1-a\Delta t_n)$.  Divide by $P_n$ and sum over $n\le m$; since $(1-a\Delta t_n)P_n=(1+b\Delta t_n)P_{n-1}\ge1$ and $P_n\le P_m$, multiplying by $P_m$ gives the claim with the factor $P_m\le\exp((a+b)t_m/(1-a\Delta t))$.
\end{proof}

\begin{proposition}[Discrete well-posedness and stability]\label{prop:stability}
Let Assumption~\ref{ass:penalty} hold and $2C_0\Delta t\le\frac12$.  Then \eqref{eq:scheme} has a unique solution at every step, and
\begin{multline*}
\max_n\|\widehat U_h^n\|_{L^2}^2+\sum_n\|\widehat U_h^n-\widehat U_h^{n-1}\|_{L^2}^2+c_0\sum_n\Delta t_n\|\widehat U_h^n\|_{DG}^2\\
\le 2e^{4C_0T}\Bigl(\|\widehat U_h^0\|_{L^2}^2+\frac1{c_0}\sum_n\Delta t_n\|\hat f^n(U_h^{n-1})\|_{L^2}^2\Bigr),
\end{multline*}
where $\|\hat f^n(w)\|_{L^2}\le\|f^n\|_{L^2}+\|\delta_tG^n\|_{L^2}+r\|G^n\|_{L^2}+|g_G(t_n)|(r+a_{\max})|\Omega|^{1/2}$ uniformly in $w$.
\end{proposition}
\begin{proof}
The bilinear form $\Delta t_n^{-1}(u,v)+\Ah(U_h^{n-1};u,v)$ is coercive on $V_h$ by Lemma~\ref{lem:garding} because $\Delta t_n^{-1}>C_0$, so the square system is invertible.  Testing \eqref{eq:lifted-scheme} with $\widehat U_h^n$, using $2(a-b,a)=\|a\|^2-\|b\|^2+\|a-b\|^2$, Lemma~\ref{lem:garding}, and Young's inequality gives the hypothesis of Lemma~\ref{lem:gronwall} with $a=2C_0$, $b=0$; the factor $2$ accounts for taking the maximum and the sums separately.
\end{proof}

\subsection{Error estimate}

\begin{theorem}[Error estimate]\label{thm:error}
Let Assumptions~\ref{ass:regularity} and \ref{ass:penalty} hold, $2C_0\Delta t\le\frac12$, and let $U_h^n$ solve \eqref{eq:scheme}.  Then
\[
\max_n\|u(t_n)-U_h^n\|_{L^2}+\Bigl(\sum_n\Delta t_n\|u(t_n)-U_h^n\|_{DG}^2\Bigr)^{1/2}
\le C\bigl(\|\Pi_hu_0-U_h^0\|_{L^2}+h^p+\Delta t\bigr).
\]
For $g_D=0$ the lifting is trivial and the proof below applies verbatim with $G=0$.
\end{theorem}
\begin{proof}
\emph{Error splitting.}  Let $\hat u=u-G$, $\widehat U_h^n=U_h^n-G^n$, $\eta^n=u^n-\Pi_hu^n$, and $\xi^n=\Pi_hu^n-U_h^n$.  Since $\Pi_h$ reproduces the affine function $G^n$, also $\eta^n=\hat u^n-\Pi_h\hat u^n$ and $\xi^n=\Pi_h\hat u^n-\widehat U_h^n$.  The trial function in all residuals below is the lifted exact solution $\hat u^n$, which vanishes on $\partial\Omega$; hence $\|\hat u^n\|_{DG,{\rm ext}}\le C\|\hat u^n\|_{H^2}$ is bounded independently of $h$.  (The full state $u^n$ has boundary jumps $g_D n_e$, and $\|u^n\|_{DG,{\rm ext}}\simeq h^{-1/2}|g_D|$; it enters only as a coefficient argument.)

\emph{Exact frozen identity.}  By Lemma~\ref{lem:recovery}\ref{it:inactive} the coefficients of \eqref{eq:lifted} at $t_n$ are $A_{u^n},\beta_{u^n},b_{u^n}$, so \eqref{eq:consistency} applied to $\hat u$ gives, for all $v\in V_h$,
\[
(\delta_t\hat u^n,v)+\Ah(u^{n-1};\hat u^n,v)=(\hat f^n(u^{n-1}),v)-(\tau^n,v)-R_{\rm fr}^n(v),
\]
with $\tau^n=\partial_tu(t_n)-\delta_tu^n$ (the backward-difference defects of $\hat u$ and $G$ combine into that of $u$) and
$R_{\rm fr}^n(v)=\Ah(u^n;\hat u^n,v)-\Ah(u^{n-1};\hat u^n,v)-g_G(t_n)(b_{u^n}-b_{u^{n-1}},v)$.

\emph{Error equation.}  Subtracting \eqref{eq:lifted-scheme} and inserting $\hat u^n-\widehat U_h^n=\eta^n+\xi^n$,
\[
(\delta_t\xi^n,v)+\Ah(U_h^{n-1};\xi^n,v)=-\Ah(U_h^{n-1};\eta^n,v)-(\tau^n,v)-R_{\rm fr}^n(v)+R_{\rm st}^n(v),
\]
where $(\delta_t\eta^n,v)=0$ because $\Pi_h$ is the $L^2$ projection and does not depend on $t$, and $R_{\rm st}^n(v)=\Ah(U_h^{n-1};\hat u^n,v)-\Ah(u^{n-1};\hat u^n,v)+g_G(t_n)(b_{u^{n-1}}-b_{U_h^{n-1}},v)$.

\emph{Residual bounds.}  Write $J_n=(t_{n-1},t_n)$.  By \eqref{eq:approx}, Taylor's formula, Lemma~\ref{lem:distance}\ref{it:lip-form},\ref{it:dist}, and $\Osw u^m=u^m$,
\begin{align*}
\|\tau^n\|_{L^2}^2&\le\Delta t_n\|\partial_{tt}u\|_{L^2(J_n;L^2)}^2,\\
|\Ah(U_h^{n-1};\eta^n,v)|&\le C\|\eta^n\|_{DG,{\rm ext}}\|v\|_{DG}\le Ch^p\|v\|_{DG},\\
|R_{\rm fr}^n(v)|&\le C\|u^n-u^{n-1}\|_{L^2}\|v\|_{DG}\le C\Delta t_n^{1/2}\|\partial_tu\|_{L^2(J_n;L^2)}\|v\|_{DG}.
\end{align*}
For the state residual, Lemma~\ref{lem:distance}, $\Osw U_h^{n-1}-u^{n-1}=-\Osw\xi^{n-1}+(\Osw\Pi_hu^{n-1}-u^{n-1})$, \eqref{eq:oswald}, and the argument of Lemma~\ref{lem:distance}\ref{it:compat} give
\[
|R_{\rm st}^n(v)|\le C_{\rm st}\bigl(\|\xi^{n-1}\|_{L^2}+h^{p+1}\bigr)\|v\|_{DG},\qquad C_{\rm st}=C\Bigl(\sup_n\|\hat u^n\|_{H^2}+\sup_t|g_G|\Bigr).
\]
Collect the state-independent residuals in
\[
\Theta_n:=Ch^p+\|\tau^n\|_{L^2}+C\Delta t_n^{1/2}\|\partial_tu\|_{L^2(J_n;L^2)}+C_{\rm st}h^{p+1},
\]
so that
\[
\sum_n\Delta t_n\Theta_n^2\le C(h^{2p}+\Delta t^2).
\]

\emph{Energy argument.}  Test with $v=\xi^n$ and use Lemma~\ref{lem:garding}:
\[
\tfrac12\bigl(\|\xi^n\|^2-\|\xi^{n-1}\|^2+\|\xi^n-\xi^{n-1}\|^2\bigr)+c_0\Delta t_n\|\xi^n\|_{DG}^2
\le C_0\Delta t_n\|\xi^n\|^2+\Delta t_n\bigl(\Theta_n+C_{\rm st}\|\xi^{n-1}\|\bigr)\|\xi^n\|_{DG}.
\]
Young's inequality with $\frac{c_0}{2}\|\xi^n\|_{DG}^2$ and multiplication by $2$ yield
$\|\xi^n\|^2-\|\xi^{n-1}\|^2+c_0\Delta t_n\|\xi^n\|_{DG}^2\le2C_0\Delta t_n\|\xi^n\|^2+\frac{4C_{\rm st}^2}{c_0}\Delta t_n\|\xi^{n-1}\|^2+\frac4{c_0}\Delta t_n\Theta_n^2$.
Lemma~\ref{lem:gronwall} with $a=2C_0\le\frac1{2\Delta t}$ and $b=4C_{\rm st}^2/c_0$ gives $\max_n\|\xi^n\|+(\sum_n\Delta t_n\|\xi^n\|_{DG}^2)^{1/2}\le C(\|\xi^0\|+h^p+\Delta t)$.  The same bound holds for $\eta^n$ by \eqref{eq:approx}, and the triangle inequality concludes.
\end{proof}

\begin{remark}[Constants and sharpness]\label{rem:constants}
The constants depend on negative powers of $\eps$ and $\ell$ through $M$, $C_{\beta\gamma}$, $C_0$, and $C_d$, and the condition $2C_0\Delta t\le\frac12$ uses the worst-case $C_0$ of Lemma~\ref{lem:garding}.  This universal constant is distinct from the run-specific quantity $C_0^{\rm obs}=\max(0,\frac12\max_n\|\gamma_{U_h^n}\|_{L^\infty}-r)$ reported in Section~\ref{sec:numerics}, which only indicates whether the mechanism operates with a moderate constant in a given computation.  The observed $L^2$ rate $h^{p+1}$ (Section~\ref{sec:mms}) would follow from a duality argument with a state-dependent elliptic projection \cite{Wheeler1973,Thomee2006}; we do not carry this out.
\end{remark}

\section{Recovery of the migration boundary}\label{sec:boundary}

Let $F(x,t)=u(x,t)-\Psi(x)$.  The numerical boundary $s_h^n$ is a root of $F_h^n=\Osw U_h^n-\Psi$ in a tracked interval $I_n$.

\begin{assumption}[Local transversality and mesh regularity]\label{ass:transversality}
For each $n$ there are an interval $I_n\subset\Omega$ and $s^n\in I_n$ with $F(s^n,t_n)=0$, ${\rm dist}(s^n,\partial I_n)\ge d_*>0$, and $|\partial_xF(\cdot,t_n)|\ge\kappa>0$ on $I_n$.  Elements meeting $I_n$, and their neighbors, satisfy $h_K\ge c_{\rm qu}h$.
\end{assumption}

The transversality is local because natural data can produce further level crossings elsewhere in $\Omega$; quasi-uniformity is required only near $I_n$.

\begin{lemma}[Local reconstruction estimate]\label{lem:reconstruction}
Under Assumption~\ref{ass:transversality},
$\|\Osw U_h^n-u^n\|_{L^\infty(I_n)}\le C\bigl(h^{-1/2}\|\xi^n\|_{L^2(\Omega)}+h^{p+1/2}|u^n|_{H^{p+1}}\bigr)$, with $C$ depending on $c_{\rm qu}$, $p$, and $\varrho$.
\end{lemma}
\begin{proof}
On an element $K$ meeting $I_n$, $\Osw U_h^n-u^n=-\Osw\xi^n+(\Osw\Pi_hu^n-\Pi_hu^n)-\eta^n$.  The inverse inequality on $K$ and the local bound \eqref{eq:oswald} give $\|\Osw\xi^n\|_{L^\infty(K)}\le Ch_K^{-1/2}\|\xi^n\|_{L^2(\omega_K)}$; by \eqref{eq:oswald}, the inverse inequality, and $\jump{\Pi_hu^n}=-\jump{\eta^n}$, the second term is bounded by $C(\sum_{x_j\in\partial K}\jump{\eta^n}_{x_j}^2)^{1/2}$, and the third by the Agmon inequality $\|\eta\|_{L^\infty(K)}^2\le C(h_K^{-1}\|\eta\|_{L^2(K)}^2+h_K\|\partial_x\eta\|_{L^2(K)}^2)$; by \eqref{eq:approx} both are $O(h^{p+1/2})$.  Use $h_K\ge c_{\rm qu}h$.
\end{proof}

\begin{theorem}[Boundary recovery]\label{thm:boundary}
Let Assumption~\ref{ass:transversality} hold.  If $\delta_n:=\|\Osw U_h^n-u^n\|_{L^\infty(I_n)}<\kappa d_*$, then $F_h^n$ has a root in $I_n$, and every root in $I_n$ satisfies $|s_h^n-s^n|\le\delta_n/\kappa$.  Under the hypotheses of Theorem~\ref{thm:error} with $U_h^0=\Pi_hu_0$,
\[
\max_n|s_h^n-s^n|\le C\kappa^{-1}h^{-1/2}(h^p+\Delta t)
\]
whenever the right-hand side is below $d_*$, which holds for $h$ small if $\Delta t=o(h^{1/2})$.
\end{theorem}
\begin{proof}
$F(\cdot,t_n)$ is strictly monotone on $I_n$ with $|F(x,t_n)|\ge\kappa|x-s^n|$, and $F_h^n-F(\cdot,t_n)=\Osw U_h^n-u^n$.  At $x=s^n\pm\delta'$, $\delta_n/\kappa<\delta'<d_*$, $F_h^n$ has the sign of $F$, so it has a root by the intermediate value theorem; if $F_h^n(s_h^n)=0$ then $\kappa|s_h^n-s^n|\le|F(s_h^n,t_n)|\le\delta_n$.  The rate follows from Lemma~\ref{lem:reconstruction} and the bound $\max_n\|\xi^n\|_{L^2}\le C(h^p+\Delta t)$ obtained in the proof of Theorem~\ref{thm:error}.
\end{proof}

The rate is suboptimal by $h^{-1/2}$ because it passes through $L^2$.  An optimal pointwise estimate $\|\Osw U_h^n-u^n\|_{L^\infty(I_n)}\le C(h^{p+1}+\Delta t)$, which we do not prove, would give $|s_h^n-s^n|\le C\kappa^{-1}(h^{p+1}+\Delta t)$.  Comparison with the boundary of the sharp-interface model requires a separate modeling estimate for $|s_\eps-s|$, which is not addressed here.

\section{Numerical experiments}\label{sec:numerics}

\paragraph{Setup.}  We take $\Omega=(-4,4)$, $T=1$, $\sigma_H=0.1$, $\sigma_L=0.3$, $r=0.05$, $\nu=0.5$, $H_\eps(s)=\frac12(1+\tanh(s/\eps))$, $\eps=0.05$, $\ell=\eps$, and uniform meshes.  The computations use exactly the operators of Section~\ref{sec:scheme}: the bump kernel $\rho(t)\propto\exp(-1/(1-t^2))$, the Hestenes extension with $c=(6,-8,3)$, the Oswald operator of Definition~\ref{def:oswald}, and the limiter \eqref{eq:recovery} with $C_{E,0}=18.96$ and $C_{S,3}=1529$, computed from these choices.  The convolution $\Sm\Osw U_h$ and its first three derivatives are evaluated at quadrature points and nodes by composite Gauss--Legendre quadrature of the exact convolution integral (ten points on cells of length $\le\ell/32$); this reproduces the closed-form mollification of quadratics to $10^{-12}$ in value and $10^{-6}$ in the second derivative.  Element integrals use $p+3$ Gauss points on subcells of length $\le\eps/4$, so that the $\eps$-layer of the coefficients is integrated accurately on coarse meshes.  The penalty is $\eta_e=72p^2$, the value of Assumption~\ref{ass:penalty} for $a_{\max}/a_{\min}=9$, and Dirichlet data are imposed by \eqref{eq:nitsche-load}.  The numerical boundary $s_h^n$ is located by a sign change of $\Osw U_h^n-\Psi$ in $I_n=[s_h^{n-1}-\frac12,s_h^{n-1}+\frac12]$ followed by Brent's method \cite{Brent1973}.

\paragraph{Diagnostics.}  For each run we report: the limiter ratio $\mu_h=\max_n\|\Sm\Osw U_h^n\|_{H^3}/(M/2)$, where $\mu_h<1$ certifies that the limiter is inactive at every step, so that the computed solution is also that of the unlimited scheme $\Sm\Osw$; the observed bound $\gamma^{\rm obs}=\max_n\|\gamma_{U_h^n}\|_{L^\infty}$ and $2C_0^{\rm obs}\Delta t$ with $C_0^{\rm obs}=\max(0,\frac12\gamma^{\rm obs}-r)$ (Remark~\ref{rem:constants}); the coercivity ratio $\vartheta_h=\min_{v\ne0}a_h(w;v,v)/(\frac12\|v\|_{A_w,h}^2)$ at the final frozen state, which Lemma~\ref{lem:garding}(a) asserts to be at least one; and, for the migration runs, $\kappa_h^{\min}=\min_n|\partial_x(\Osw U_h^n-\Psi)(s_h^n)|$ and the number of sign changes of $\Osw U_h^n-\Psi$ in $\Omega$.  A positive, mesh-independent $\kappa_h^{\min}$ with a single sign change is numerical evidence consistent with Assumption~\ref{ass:transversality}; it does not verify the assumption, which concerns the exact solution on the whole interval $I_n$.

\subsection{Manufactured solutions}\label{sec:mms}

For the spatial study we use the stationary solution $u_{\rm ex}(x)=1+\sin(\pi(x-x_{\min})/L)$ with $g_D=1$ and the source $f=-a_\eps(\Sm u_{\rm ex})u_{\rm ex}''-b_\eps(\Sm u_{\rm ex})u_{\rm ex}'+ru_{\rm ex}$, where $\Sm u_{\rm ex}$ is evaluated by the same quadrature as in the scheme, so that $u_{\rm ex}$ solves \eqref{eq:model}.  Since $u_{\rm ex}$ does not depend on $t$, the residuals $\tau^n$ and $R_{\rm fr}^n$ in the proof of Theorem~\ref{thm:error} vanish, and the computed error is purely spatial for every $\Delta t$; we use $\Delta t=0.02$ and $U_h^0=\Pi_hu_{\rm ex}$.  Table~\ref{tab:mms-spatial} reports $E_{L^2}=\|u_{\rm ex}-U_h^N\|_{L^2}$, the energy error $E_{DG}=(\sum_n\Delta t\|u_{\rm ex}-U_h^n\|_{DG}^2)^{1/2}$ of Theorem~\ref{thm:error}, and the compatibility quantities of Lemma~\ref{lem:distance}.  The energy error converges at the predicted rate $p$.  The $L^2$ error converges at rate $p+1$ (Remark~\ref{rem:constants}); for $p=1$ the observed orders oscillate on the intermediate meshes before settling at two.  The recovery-input error $e_{\rm pc}=\|\Osw\Pi_hu_{\rm ex}-u_{\rm ex}\|_{L^2}$ converges at the rate $h^{p+1}$ of Lemma~\ref{lem:distance}\ref{it:compat}, and the ratio $\chi_h/e_{\rm pc}$, with $\chi_h=d^\sharp(\Pi_hu_{\rm ex},u_{\rm ex})$, stays bounded, as asserted by Lemma~\ref{lem:distance}\ref{it:dist}.  The ratio is not monotone: $e_{\rm pc}$ oscillates on the mesh scale, and these oscillations are damped in $\partial_x^2\Sm$ only by the Fourier transform of $\rho$ at frequency $\sim\ell/h$, which decays faster than any power but slowly at moderate $\ell/h$.  We therefore report the bounded ratio rather than an empirical order of $\chi_h$.

\begin{table}[htbp]\centering\small
\caption{Spatial convergence for the stationary manufactured solution ($\eps=\ell=0.05$, $\Delta t=0.02$, $T=1$).  $E_{DG}$ is the energy error of Theorem~\ref{thm:error}; $e_{\rm pc}=\|\Osw\Pi_hu_{\rm ex}-u_{\rm ex}\|_{L^2}$ and $\chi_h=d^\sharp(\Pi_hu_{\rm ex},u_{\rm ex})$ (Lemma~\ref{lem:distance}).}
\label{tab:mms-spatial}
\begin{tabular}{rrccccccr}\toprule
$p$ & $N_h$ & $E_{L^2}$ & EOC & $E_{DG}$ & EOC & $e_{\rm pc}$ & EOC & $\chi_h/e_{\rm pc}$\\\midrule
1 & 20 & 1.96e-3 & -- & 3.59e-2 & -- & 1.84e-3 & -- & 430 \\
1 & 40 & 5.95e-4 & 1.72 & 1.80e-2 & 1.00 & 4.60e-4 & 2.00 & 805 \\
1 & 80 & 1.24e-4 & 2.27 & 8.93e-3 & 1.01 & 1.15e-4 & 2.00 & 656 \\
1 & 160 & 4.03e-5 & 1.62 & 4.46e-3 & 1.00 & 2.87e-5 & 2.00 & 430 \\
1 & 320 & 1.01e-5 & 2.00 & 2.23e-3 & 1.00 & 7.18e-6 & 2.00 & 51 \\
1 & 640 & 2.60e-6 & 1.96 & 1.11e-3 & 1.00 & 1.80e-6 & 2.00 & 196 \\
\midrule
2 & 20 & 4.66e-5 & -- & 7.71e-4 & -- & 4.46e-5 & -- & 483 \\
2 & 40 & 5.70e-6 & 3.03 & 1.86e-4 & 2.05 & 5.57e-6 & 3.00 & 400 \\
2 & 80 & 7.07e-7 & 3.01 & 4.58e-5 & 2.02 & 6.97e-7 & 3.00 & 478 \\
2 & 160 & 8.66e-8 & 3.03 & 1.14e-5 & 2.01 & 8.71e-8 & 3.00 & 326 \\
2 & 320 & 1.11e-8 & 2.97 & 2.85e-6 & 2.00 & 1.09e-8 & 3.00 & 44 \\
\midrule
3 & 20 & 4.30e-7 & -- & 9.99e-6 & -- & 2.43e-7 & -- & 522 \\
3 & 40 & 3.49e-8 & 3.62 & 1.24e-6 & 3.01 & 1.51e-8 & 4.01 & 752 \\
3 & 80 & 1.89e-9 & 4.21 & 1.51e-7 & 3.04 & 9.44e-10 & 4.00 & 462 \\
3 & 160 & 1.05e-10 & 4.17 & 1.88e-8 & 3.00 & 5.90e-11 & 4.00 & 224 \\
\bottomrule\end{tabular}\end{table}

For the temporal study we use $u_{\rm ex}(x,t)=1+\sin(\pi(x-x_{\min})/L)e^{-t}$ with the corresponding source, $p=3$, and $N_h=160$, for which the spatial error is negligible.  Table~\ref{tab:mms-temporal} shows first-order convergence in both norms.  At the three largest steps the observed product $2C_0^{\rm obs}\Delta t$ exceeds $\frac12$ (2.95, 1.48, 0.74), so the scheme remains stable outside the sufficient condition $2C_0\Delta t\le\frac12$ of Proposition~\ref{prop:stability}, even with the observed rather than the worst-case constant.

\begin{table}[htbp]\centering\small
\caption{Temporal convergence for the time-dependent manufactured solution ($p=3$, $N_h=160$, $\eps=\ell=0.05$, $T=1$).}
\label{tab:mms-temporal}
\begin{tabular}{rcccccc}\toprule
$N_t$ & $\Delta t$ & $E_{L^2}$ & EOC & $E_{DG}$ & EOC & $2C_0^{\rm obs}\Delta t$\\\midrule
10 & 1.00e-1 & 5.99e-2 & -- & 4.46e-2 & -- & 2.95 \\
20 & 5.00e-2 & 3.02e-2 & 0.99 & 2.19e-2 & 1.03 & 1.48 \\
40 & 2.50e-2 & 1.52e-2 & 0.99 & 1.09e-2 & 1.01 & 0.74 \\
80 & 1.25e-2 & 7.62e-3 & 1.00 & 5.40e-3 & 1.01 & 0.37 \\
160 & 6.25e-3 & 3.81e-3 & 1.00 & 2.70e-3 & 1.00 & 0.18 \\
320 & 3.13e-3 & 1.91e-3 & 1.00 & 1.35e-3 & 1.00 & 0.09 \\
640 & 1.56e-3 & 9.54e-4 & 1.00 & 6.73e-4 & 1.00 & 0.05 \\
\bottomrule\end{tabular}\end{table}

\subsection{Migration problem}\label{sec:migration}

We take $f=0$, $g_D(x_{\min},t)=e^{x_{\min}}$, $g_D(x_{\max},t)=e^{-rt}$, and the smoothed payoff $u_0=e^x(1-\widetilde H_\eps)+\widetilde H_\eps$.  The boundary values are the deep-default value (bond value $\approx$ firm value) and the discounted face value.  Here $\widetilde H_\eps(x)=\zeta(x)H_\eps(x)+(1-\zeta(x))\mathbf 1_{(0,\infty)}(x)$, with the switch $H_\eps$ evaluated at $s=x$, with a smooth cutoff $\zeta$ equal to one on $[x_{\min}+1,x_{\max}-1]$ and vanishing near $\partial\Omega$; $|\widetilde H_\eps-H_\eps|$ and its derivatives are below $10^{-50}$ for $\eps=0.05$, so $\widetilde H_\eps$ and $H_\eps$ coincide in floating point.  The level function $u_0-\Psi$ has the single crossing $x\approx\ln2$ with slope $\approx-1$.

\begin{remark}[Compatibility of the migration data]\label{rem:payoff-compat}
Since $b_\eps=r-a_\eps$, \eqref{eq:model} with $f=0$ reads $\partial_tu=a_\eps(\Sm u)(u_{xx}-u_x)+r(u_x-u)$.  The datum equals $e^x$ near $x_{\min}$, where $u_0''-u_0'=u_0'-u_0=0$, and $1$ near $x_{\max}$.  By induction, whatever the nonlocal coefficient, $\partial_t^ku(\cdot,0)$ vanishes near $x_{\min}$ and equals $(-r)^k$ near $x_{\max}$ for $k\ge1$, on neighborhoods that do not depend on $k$.  These values match $\partial_t^kg_D$ at both corners, so the compatibility conditions hold exactly at every order.  By contrast, the data $g_D=e^{-rt}\min(1,e^x)$ satisfy the zeroth-order condition but violate the first-order condition at $x_{\min}$ by $re^{x_{\min}}$.
\end{remark}

No closed-form solution is available.  The reference solution $u_{\rm ref}$ is computed with the same scheme, $p=2$, and $N_h^{\rm ref}=1280$, and its time-discretization error is removed to leading order by Richardson extrapolation of backward Euler runs with $N_t=12800$ and $25600$.  Its accuracy was checked by two independent changes: replacing $N_h^{\rm ref}$ by $2560$ changes $u_{\rm ref}(T)$ by $2.94\cdot10^{-9}$ in $L^2$ and the boundary trajectory by $1.39\cdot10^{-9}$; replacing the pair of time steps by $(6400,12800)$ changes them by $3.77\cdot10^{-11}$ and $2.43\cdot10^{-11}$.  Both are at least an order of magnitude below the errors in Table~\ref{tab:migration}.  The compared runs use $p=2$ and $\Delta t=h^2/4$, so that both terms of Theorem~\ref{thm:error} are $O(h^2)$.  We report $E_{L^2}=\|u_{\rm ref}(T)-U_h^N\|_{L^2}$, $E_{L^\infty}=\|u_{\rm ref}(T)-\Osw U_h^N\|_{L^\infty}$, and $E_s=\max_n|s_h^n-s_{\rm ref}(t_n)|$.  Since the reference uses the same discretization, this is a self-convergence study.

\begin{table}[htbp]\centering\small
\caption{Migration problem: errors against the reference solution ($p=2$, $\eps=\ell=0.05$, $\Delta t=h^2/4$, $T=1$), minimum numerical transversality, and maximal number of sign changes of $\Osw U_h^n-\Psi$ in $\Omega$.}
\label{tab:migration}
\begin{tabular}{rrcccccccrc}\toprule
$N_h$ & $N_t$ & $h/\eps$ & $E_{L^2}$ & EOC & $E_{L^\infty}$ & EOC & $E_s$ & EOC & $\kappa_h^{\min}$ & roots\\\midrule
40 & 100 & 4.00 & 1.27e-4 & -- & 2.22e-4 & -- & 5.44e-5 & -- & 0.9354 & 1\\
80 & 400 & 2.00 & 2.78e-5 & 2.19 & 4.61e-5 & 2.27 & 7.43e-6 & 2.87 & 0.9359 & 1\\
160 & 1600 & 1.00 & 6.69e-6 & 2.05 & 1.09e-5 & 2.08 & 1.46e-6 & 2.35 & 0.9357 & 1\\
320 & 6400 & 0.50 & 1.66e-6 & 2.01 & 2.69e-6 & 2.02 & 3.28e-7 & 2.16 & 0.9358 & 1\\
\bottomrule\end{tabular}\end{table}

All errors in Table~\ref{tab:migration} decrease at second order, as predicted for $\Delta t\sim h^2$, including the pairs with $h/\eps=4$ and $2$.  The boundary error decays at least as fast as the $L^\infty$ reconstruction error, i.e.\ faster than the proved bound $O(h^{3/2})$ of Theorem~\ref{thm:boundary} for this coupling, and in agreement with the conditional rate $O(h^{p+1}+\Delta t)$.  The crossing remains unique and transversal throughout, with $\kappa_h^{\min}\approx0.94$.  The boundary moves from $\ln2\approx0.693$ to $s(T)\approx0.642$, close to $\ln2-rT$: away from the threshold the solution behaves like the discounted face value, and the crossing follows the discounting.  Figure~\ref{fig:migration} shows the reconstruction near the crossing and the boundary trajectories.

\subsection{Sensitivity to the regularization and diagnostics}\label{sec:diagnostics}

Table~\ref{tab:diagnostics} collects the diagnostics for representative runs: the spatial MMS run with $p=2$, $N_h=160$; the migration run with $N_h=320$; and migration runs with $\eps\in\{0.1,0.05,0.025,0.0125\}$, $\ell=\eps$, $h=\eps/4$, and $N_t=1000$.  In every run the limiter ratio is below $10^{-5}$, so the analyzed limited scheme and the unlimited scheme coincide; this also shows that the data-computable $M$ is very conservative.  The coercivity ratio is about $1.75$, consistent with the sufficient penalty $\eta_0=72$.  The observed coefficient bound $\gamma^{\rm obs}$ grows like $\eps^{-2}$, as expected from the term $\partial_z^2a_\eps\sim H_\eps''$ in $\gamma_w$, while the boundary location and the transversality change by less than $10^{-3}$ and $2\%$, respectively, over this range.  For $\eps=0.0125$ the product $2C_0^{\rm obs}\Delta t$ is already $0.11$ at $\Delta t=10^{-3}$, which illustrates why time steps must decrease with $\eps$.  These experiments support the fixed-regularization analysis; they do not address the sharp-interface limit.

\begin{table}[htbp]\centering\small
\caption{Diagnostics: limiter ratio $\mu_h$, observed coefficient bound $\gamma^{\rm obs}$, $2C_0^{\rm obs}\Delta t$, coercivity ratio $\vartheta_h$ (not computed for the two finest meshes), final boundary location, and minimum transversality.  The last four rows use $\ell=\eps$, $h=\eps/4$, $N_t=1000$.}
\label{tab:diagnostics}
\begin{tabular}{lcrccrcccc}\toprule
Run & $\eps$ & $N_h$ & $\Delta t$ & $\mu_h$ & $\gamma^{\rm obs}$ & $2C_0^{\rm obs}\Delta t$ & $\vartheta_h$ & $s_h(T)$ & $\kappa_h^{\min}$\\\midrule
MMS, $p=2$ & 0.05 & 160 & 2.00e-2 & 5.62e-8 & 29.8 & 0.59 & 1.752 & -- & --\\
Migration & 0.05 & 320 & 1.56e-4 & 1.85e-6 & 6.9 & 0.001 & 1.752 & 0.6423 & 0.9358\\
\midrule
Migration & 0.1 & 320 & 1.00e-3 & 5.22e-6 & 1.7 & 0.002 & 1.752 & 0.6428 & 0.9426\\
Migration & 0.05 & 640 & 1.00e-3 & 1.85e-6 & 6.9 & 0.007 & 1.752 & 0.6423 & 0.9357\\
Migration & 0.025 & 1280 & 1.00e-3 & 6.55e-7 & 28.0 & 0.028 & -- & 0.6423 & 0.9329\\
Migration & 0.0125 & 2560 & 1.00e-3 & 2.32e-7 & 113.1 & 0.113 & -- & 0.6423 & 0.9315\\
\bottomrule\end{tabular}\end{table}

\begin{figure}[htbp]\centering
\includegraphics[width=\textwidth]{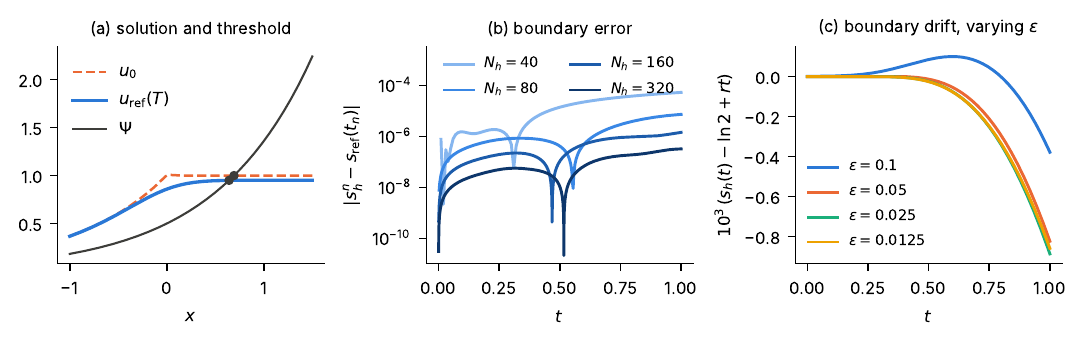}
\caption{Migration problem ($p=2$, $\ell=\eps$).  (a) Initial datum, reference solution at $T=1$, and threshold $\Psi$ near the crossing ($\eps=0.05$); dots mark the boundary at $t=0$ and $t=T$.  (b) Boundary error against the reference for the runs of Table~\ref{tab:migration}.  (c) Deviation of the computed boundary from $\ln2-rt$ for the $\eps$-runs of Table~\ref{tab:diagnostics}.}
\label{fig:migration}
\end{figure}

\section{Conclusion}\label{sec:conclusion}

We constructed a recovered-coefficient SIPG/upwind method for a fixed-$(\eps,\ell)$ averaged-state regularization of the credit-rating migration equation.  Evaluating the coefficients through Oswald averaging, the model mollifier, and a data-calibrated $H^3$ cutoff makes the coefficients single-valued at DG interfaces and turns admissibility, $L^2$-Lipschitz stability, and $h^{p+1}$ projection compatibility of the recovery into proved properties.  For fixed regularization the frozen backward Euler scheme is stable and converges at order $O(h^p+\Delta t)$ in the DG energy norm, including the nonhomogeneous Dirichlet data of the migration problem; the regularized boundary is recovered at a proved suboptimal rate under local transversality.  Computations with the analyzed operator confirm these rates.  Open questions are optimal $L^2$ and $L^\infty$ estimates, a regularity theorem with sharp compatibility conditions, the joint limit $(\eps,\ell)\to(0,0)$ and the corresponding boundary modeling error, and extensions to two space dimensions.


\begin{thebibliography}{99}\small

\bibitem{AdamsFournier2003}
R.~A. Adams and J.~J.~F. Fournier, \emph{Sobolev Spaces}, 2nd ed., Academic Press, Amsterdam, 2003.

\bibitem{Arnold2002}
D.~N. Arnold, F.~Brezzi, B.~Cockburn, and L.~D. Marini, Unified analysis of discontinuous Galerkin methods for elliptic problems, \emph{SIAM J. Numer. Anal.} 39 (2002), 1749--1779.

\bibitem{Brenner2003}
S.~C. Brenner, Poincar\'e--Friedrichs inequalities for piecewise $H^1$ functions, \emph{SIAM J. Numer. Anal.} 41 (2003), 306--324.

\bibitem{Brent1973}
R.~P. Brent, \emph{Algorithms for Minimization without Derivatives}, Prentice-Hall, Englewood Cliffs, NJ, 1973.

\bibitem{BrezziMariniSuli2004}
F.~Brezzi, L.~D. Marini, and E.~S\"uli, Discontinuous Galerkin methods for first-order hyperbolic problems, \emph{Math. Models Methods Appl. Sci.} 14 (2004), 1893--1903.

\bibitem{BurmanErn2007}
E.~Burman and A.~Ern, Continuous interior penalty $hp$-finite element methods for advection and advection-diffusion equations, \emph{Math. Comp.} 76 (2007), 1119--1140.

\bibitem{CockburnShu2001}
B.~Cockburn and C.-W. Shu, Runge--Kutta discontinuous Galerkin methods for convection-dominated problems, \emph{J. Sci. Comput.} 16 (2001), 173--261.

\bibitem{DiPietroErn2012}
D.~A. Di~Pietro and A.~Ern, \emph{Mathematical Aspects of Discontinuous Galerkin Methods}, Springer, Berlin, 2012.

\bibitem{DouglasDupont1970}
J.~Douglas, Jr.\ and T.~Dupont, Galerkin methods for parabolic equations, \emph{SIAM J. Numer. Anal.} 7 (1970), 575--626.

\bibitem{Hestenes1941}
M.~R. Hestenes, Extension of the range of a differentiable function, \emph{Duke Math. J.} 8 (1941), 183--192.

\bibitem{HuLiangWu2015}
B.~Hu, J.~Liang, and Y.~Wu, A free boundary problem for corporate bond with credit rating migration, \emph{J. Math. Anal. Appl.} 428 (2015), 896--909.

\bibitem{KarakashianPascal2003}
O.~A. Karakashian and F.~Pascal, A posteriori error estimates for a discontinuous Galerkin approximation of second-order elliptic problems, \emph{SIAM J. Numer. Anal.} 41 (2003), 2374--2399.

\bibitem{LSU1968}
O.~A. Ladyzhenskaya, V.~A. Solonnikov, and N.~N. Uraltseva, \emph{Linear and Quasilinear Equations of Parabolic Type}, AMS, Providence, RI, 1968.

\bibitem{Lieberman1996}
G.~M. Lieberman, \emph{Second Order Parabolic Differential Equations}, World Scientific, Singapore, 1996.

\bibitem{Riviere2008}
B.~Rivi\`ere, \emph{Discontinuous Galerkin Methods for Solving Elliptic and Parabolic Equations: Theory and Implementation}, SIAM, Philadelphia, 2008.

\bibitem{Thomee2006}
V.~Thom\'ee, \emph{Galerkin Finite Element Methods for Parabolic Problems}, 2nd ed., Springer, Berlin, 2006.

\bibitem{Wheeler1973}
M.~F. Wheeler, A priori $L_2$ error estimates for Galerkin approximations to parabolic partial differential equations, \emph{SIAM J. Numer. Anal.} 10 (1973), 723--759.

\bibitem{WuLiang2020}
Y.~Wu and J.~Liang, Free boundaries of credit rating migration in switching macro regions, \emph{Math. Control Relat. Fields} 10 (2020), 257--274.

\end{thebibliography}
\end{document}